\documentclass[a4paper]{amsart}
\usepackage{times}
\usepackage{cite}
\usepackage{amsthm,amsmath,amssymb,amsfonts}
\usepackage{algorithmic}
\usepackage{graphicx}
\usepackage{textcomp}
\usepackage{mathtools, dsfont}

\usepackage{ stmaryrd }

\usepackage{bussproofs}

\usepackage{ mathrsfs }

\usepackage{relsize}

\usepackage{amssymb}
\usepackage{graphicx}

\usepackage{amsmath}
\usepackage{placeins}
\usepackage{array}

\newcommand{\vsim}{\mathrel{\scalebox{1}[1.5]{$\shortmid$}\mkern-3.1mu\raisebox{0.15ex}{$\sim$}}}

\usepackage{hyperref}
\usepackage{bussproofs}

\usepackage{tikz}
\usetikzlibrary{arrows,chains,matrix,positioning,scopes}
\def\BibTeX{{\rm B\kern-.05em{\sc i\kern-.025em b}\kern-.08em
    T\kern-.1667em\lower.7ex\hbox{E}\kern-.125emX}}

\makeatletter
\tikzset{join/.code=\tikzset{after node path={%
\ifx\tikzchainprevious\pgfutil@empty\else(\tikzchainprevious)%
edge[every join]#1(\tikzchaincurrent)\fi}}}
\makeatother

\tikzset{>=stealth',every on chain/.append style={join},
         every join/.style={->}}
\tikzstyle{labeled}=[execute at begin node=$\scriptstyle,
   execute at end node=$]

\newtheorem{Def}{Definition}
\newtheorem{Thm}{Theorem}
\newtheorem{Rmk}[Thm]{Remark}
\newtheorem{Exm}[Thm]{Example}
\newtheorem{Cor}[Thm]{Corollary}

\newtheorem{Pro}[Thm]{Proposition}

\usepackage{color,dsfont}
\theoremstyle{plain}

\theoremstyle{definition}
\usepackage[]{graphicx}

\title{ First-order friendliness}

 \author{Guillermo Badia }
 \author{David Makinson}
\address{University of Queensland, Brisbane, Australia}

 \email{g.badia@uq.edu.au, d.makinson@uq.edu.au}

 \usepackage{fancyhdr}

\begin{document}

%\date{}

\begin{abstract}
In this note we study a  counterpart in predicate logic of the notion of \emph{logical friendliness}, introduced into propositional logic in Makinson (2007). The result is a new inference relation for predicate languages using first-order models. Although  compactness and interpolation fail dramatically, several other properties are  preserved from the propositional case.

\medskip

 \medskip

 \end{abstract}

\dedicatory{This article is dedicated to our friend John N. Crossley on the occasion of his 85th birthday.}

\maketitle

\section{Introduction and definition}

The relation of logical friendliness, introduced in the propositional context in \cite{Makinson}, has a very
straightforward definition as a $\forall \exists$ version of the fundamental $\forall \forall$ notion of consequence.
Specifically, if $\Gamma$ is a set of formulae of classical propositional logic and $\phi$ is a formula of the same, $\Gamma$ is said to be friendly to $\phi$ iff for every valuation $v$ on the propositional variables occurring in
formulae of $\Gamma$, if $v(\gamma)=1$ for all $\gamma \in \Gamma$ then there is an extension of $v$ to a valuation $v'$ covering
also any remaining variables in $\phi$ such that $v(\phi) = 1$. It is thus a weakening of classical
consequence and if the existential quantifier in its definition is replaced by a universal one, it
reverts to the classical relation.

So defined, friendliness has a number of interesting features. While lacking some familiar
properties of classical consequence, it satisfies some others in full, as well as yielding `local'
versions of yet others, as shown in \cite{Makinson}. However, if we seek to extend the definition from the
propositional to the first-order context, a number of options arise due to the greater complexity
of the notion of a first-order model, with its ingredients of domain of discourse, values for
individual constants, values for predicate and function letters, and the equality relation. The various options generate distinct relations, which differ in their behaviour. {Indeed, two of the lessons of the present paper are that the concept of friendliness is less robust in the first-order context than in the propositional one and that even the seemingly best behaved of the possible first-order options is less regular than its propositional counterpart, notable with respect to compactness and interpolation.}

 Definition \ref{friend} of the present
paper chooses one of those options for close study, selected because it
manifests the greatest regularity and continuity with the propositional case. In the main text,
alternative definitions are mentioned only in passing where pointed comment is particularly
helpful, but an Appendix systematically compares them with the nominated one and each other. 

{
Attention to differences between options is important if one is to study the connections between
friendliness and cognate notions that are already familiar in logic and the philosophy of science. In
\cite{Makinson}, links between friendliness in the propositional context and concepts such the 'development'
of a propositional formula \cite{Boole} and `forgetting' a propositional letter \cite{Weber} were
identified and studied. But links to the criterion of conservativity in the theory of first-order
definition \cite{Les}, as well as to the Ramsey eliminability of terms as studied in the
philosophy of science by \cite{Van Benthem}, \cite{Rantala} and others, were merely sketched
due to the lack of a clear view of the landscape of options for friendliness in the first-order context.
The present text provides that view, on the basis of which such connections may be investigated
more thoroughly in the future.}

The following familiar notions and notations of first-order logic are background to the paper.
Let $L$ be a predicate \emph{language} (in the sense of a vocabulary of non-logical symbols) with equality (taken as a constant predicate having the fixed interpretation of `true identity' in all models), function, relation, and constant symbols. As usual in the presence of equality, we can treat function and individual constant symbols as additional predicate letters, so we may assume the language is purely relational for simplicity.  If $\phi$ is a formula of  $L$, we will refer to the sublanguage of $L$ restricted to the vocabulary of $\phi$ as $L_\phi$ (similarly, if $\Gamma \cup\{\phi\}$ is a set of formulas, we write $L_{\Gamma, \phi}$).{ If $\mathfrak{A}$ and $\mathfrak{B}$ are structures for any language  $L' $ s.t. $L_\phi \subseteq L' \subseteq L$, we  write $\mathfrak{A}\preceq_{L_\phi}\mathfrak{B}$} to mean that there is an embedding $f: A\longrightarrow
 B$ such that for any formula $\psi$ of the language $L_\phi$, $\mathfrak{A}\models \psi[\overline{a}]$ 
 iff $\mathfrak{B}\models \psi[f(\overline{a})]$ (this is an \emph{elementary} embedding). { In what follows we will use the symbol $\vdash$ to denote the usual logical consequence relation in first-order logic.}
 
Given a structure $\mathfrak{A}$ for a language $L$,  an \emph{expansion} of $\mathfrak{A}$ to a language $L' \supseteq L$ is any  model  $\mathfrak{A}'$ for  $L'$ with the same domain as  $\mathfrak{A} $, interpreting the symbols of $L$ as in $\mathfrak{A}$,   and specifying interpretations for the new predicate symbols in $L'$. We are now ready to state  our first-order counterpart of the main definition from \cite{Makinson}:

\begin{Def} \label{friend}\emph{Let $\Gamma$ be a set of sentences and $\phi$ a sentence from $L$. We say that $\Gamma$ is \emph{friendly to} $\phi$ (in symbols, $\Gamma \vsim \phi$) iff for every model $\mathfrak{A} \models \Gamma$ for the language $L_\Gamma$,  there is a model  $\mathfrak{A}'$ for the same language such that $\mathfrak{A}\preceq_{L_\Gamma}
 \mathfrak{A}'$ and, furthermore,  $\mathfrak{A}'$ can be expanded  to a model $\mathfrak{A}''$ for the language $L_{\Gamma, \phi}$ for which $\mathfrak{A}''\models \phi$}. 

\end{Def}

\noindent In the Appendix, we  show that in Definition \ref{friend} above we could replace the relation of elementary embedding by that of elementary equivalence, thus using a relation between the models  $\mathfrak{A}$ and $\mathfrak{A}'$ that is symmetric.  A simple example taken from \cite{Kossak}, where it is used for a different purpose,  shows that  $\vsim $ is not a trivial relation:

\begin{Exm}\emph{
Suppose $(G, +)$ is a six-element group. Then $\text{Th}(G, +)$, the complete first-order theory of the structure $(G, +)$, is not friendly to the sentence $\phi_{\text{Field}}$, the conjunction of the axioms for fields.  This is because every elementary extension of $(G, +)$ still has exactly six elements, and all finite fields are of order $p^n$ for some $n\geq1$ and $p$  prime.  Similarly, considering now infinite structures, the complete first-order theory of $(\mathds{Z}, +)$ is not friendly to $\phi_{\text{Field}}$ as we have that $\phi_{\text{Field}} \vdash (1+1=0) \vee \exists x (x+x=1)$  but $(1+1=0) \vee \exists x (x+x=1)$ is  false in the structure $(\mathds{Z}, +)$.}
\end{Exm}

\section{Continuities with the propositional case}

All results of this section extend to the first-order context properties established in \cite{Makinson} for the propositional one. Properties that do not  extend to the first-order context are considered in Section \ref{dis}.

The first proposition (supraclassicality) is trivial but worth stating as it is repeatedly applied:

\begin{Pro}
[Supraclassicality]\label{supra} $\Gamma \vdash \phi$ only if $\Gamma \vsim \phi$.
\end{Pro}

\begin{proof} {Suppose that $\Gamma \vdash \phi$.  Consider a model $\mathfrak{A} \models \Gamma$ for the language $L_\Gamma$. Then any model  $\mathfrak{A}'$ for the same language such that $\mathfrak{A}\preceq_{L_\Gamma}
 \mathfrak{A}'$ is a model of $\Gamma$ and, furthermore, in whatever way one expands $\mathfrak{A}'$  to a model $\mathfrak{A}''$ for the language $L_{\Gamma, \phi}$ we 
 must have  $\mathfrak{A}''\models \phi$ since $\Gamma \vdash \phi$.}
\end{proof}

Next we identify two contexts in which $\vsim$ reduces to $\vdash$:

\begin{Pro}
[First Reduction Case] \label{red1}  If $L_\phi \subseteq L_\Gamma$, then   $\Gamma \vsim \phi$ iff  $\Gamma  \vdash  \phi$.
\end{Pro}
\begin{proof} {Right to left: holds directly by supraclassicality. 
Left to right: Suppose that $\Gamma \vsim \phi$ and  $\mathfrak{A}\models \Gamma$; we need to show that $\mathfrak{A} \models  \phi$. By hypothesis,  there is a model  $\mathfrak{A}'$ for the same language such that $\mathfrak{A}\preceq_{L_\Gamma}
 \mathfrak{A}'$ and, furthermore,  $\mathfrak{A}'$ can be expanded to a model $\mathfrak{A}''$ for the language $L_{\Gamma,\phi}$ for which $\mathfrak{A}''\models \phi$. However, since $L_\phi \subseteq L_\Gamma$,   $\mathfrak{A}''=\mathfrak{A}'$, so  $\mathfrak{A}'\models \phi$ and since $\mathfrak{A}\preceq_{L_\Gamma}
 \mathfrak{A}'$, also $\mathfrak{A} \models \phi$ as desired.}
\end{proof}

In the proof of the next two propositions we will use Robinson diagrams. The reader can consult \cite{Hodges} for an  exposition of the  technique. For any structure $\mathfrak{B}$ for a language $L$,  we will let  $\text{eldiag}(\mathfrak{B})$ denote the \emph{elementary diagram} of the structure $\mathfrak{B}$. Importantly, $\text{eldiag}(\mathfrak{B})$ is a theory in the language resulting from adding to $L$ new constants for all the elements in the domain of $\mathfrak{B}$ and  containing all the first-order sentences true in  the expansion $\mathfrak{B}'$ of  $\mathfrak{B}$ where each constant has been interpreted as  its corresponding element in the domain of $\mathfrak{B}$.   The useful feature of $\text{eldiag}(\mathfrak{B})$ is that whenever a model satisfies it, $\mathfrak{B}$ can be elementarily embedded into its restriction  to language $L$ \cite[Lemma 2.5.3]{Hodges}. 

\begin{Pro}
[Second Reduction Case] \label{red2}Let $\Gamma$ be a  consistent and complete theory in a language $L_\Gamma$. Then for any sentence $\phi$  of $L$, $\Gamma \vsim \phi$ iff $\Gamma \not \vdash \neg \phi$.
\end{Pro}

\begin{proof} Let $\Gamma \vsim \phi$. Since $\Gamma$ is consistent, we have a model $\mathfrak{A} \models \Gamma$ for the language $L_\Gamma$, but given that $\Gamma \vsim \phi$, there is a model  $\mathfrak{A}'$ for the same language such that $\mathfrak{A}\preceq_{L_\Gamma}
 \mathfrak{A}'$ and, furthermore,  $\mathfrak{A}'$ can be expanded to a model $\mathfrak{A}''$ for the language $L_{\Gamma,\phi}$ for which $\mathfrak{A}''\models \phi$. But if $\Gamma \vdash \neg \phi$, then each of $\mathfrak{A}, \mathfrak{A}', \mathfrak{A}'' \models \neg \phi$. Consequently, $\Gamma \not \vdash \neg \phi$.

  For the converse, assume that $\Gamma \not \vdash \neg \phi$. Then we have a model $\mathfrak{B}$ for the language $L_{\Gamma,\phi}$ such that $\mathfrak{B}\models \Gamma$ and $\mathfrak{B}\not \models \neg \phi$, i.e., $\mathfrak{B} \models \phi$, and  clearly, $(\mathfrak{B}\upharpoonright L_\Gamma)\models \Gamma$. We need to show that $\Gamma \vsim \phi$. Let  $\mathfrak{A}$ be a model for the language $L_\Gamma$
  with $\mathfrak{A}\models \Gamma$. We need to find a model  $\mathfrak{A}'$ for the same language such that $\mathfrak{A}\preceq_{L_\Gamma}
 \mathfrak{A}'$ and an expansion $\mathfrak{A}''$ of $\mathfrak{A}'$ to the language $L_{\Gamma, \phi}$ for which $\mathfrak{A}''\models \phi$.  To find a suitable $\mathfrak{A}'$, first observe that  $\mathfrak{B} \equiv_{L_\Gamma} \mathfrak{A}$ (the structures are elementarily equivalent in the language $L_\Gamma$) since $\Gamma $ is complete. This implies that $\text{Th}_{L_\Gamma}(\mathfrak{A}) \cup \{\phi\} = \text{Th}_{L_\Gamma}(\mathfrak{B}) \cup \{\phi\}$ is consistent. By a routine compactness argument using Robinson diagrams,  there is a model  $\mathfrak{A}'$ with the properties that we need. We briefly sketch the argument.  All we need to do is show that $\text{eldiag}(\mathfrak{A}) \cup \{\phi\}$ is consistent and therefore has a model $\mathfrak{C}$ for then we can let $\mathfrak{A}'' = (\mathfrak{C} \upharpoonright L_{\Gamma,\phi})$ and $\mathfrak{A}' = (\mathfrak{C} \upharpoonright L_{\Gamma})$. Suppose otherwise, that is,  for some finite $T \subseteq \text{eldiag}(\mathfrak{A})$,  the theory $T \cup \{\phi\}$ is not consistent. Then $\phi \vdash \neg \bigwedge T$, but then since the diagram constants in $T$ are all new to $L_{\Gamma, \phi}$, $\phi \vdash \forall \overline{x}\neg \bigwedge T^*$ by \cite[Lemma 2.3.2]{Hodges} where $T^*$ is the result of replacing in every sentence in $T$ the new constants by new variables. But then $\mathfrak{B}\models  \forall \overline{x}\neg \bigwedge T^*$ and since  $\mathfrak{A} \equiv_{L_\Gamma} \mathfrak{B}$, we have that $\mathfrak{A}\models  \forall \overline{x}\neg \bigwedge T^*$ which is a contradiction since $\mathfrak{A}\models  \exists \overline{x} \bigwedge T^*$.  

\end{proof}

{The following two reduction cases may appear less interesting in themselves, but turn out to be useful in proving results on recursive enumerability (Proposition \ref{nonax} and Remark \ref{nonax3} below).}

{
\begin{Pro}
[Third Reduction Case] \label{red3}For any sentence $\phi$  of $L$ not involving the equality symbol, $ \vsim \phi$ iff $\phi$ is satisfiable.
\end{Pro}
\begin{proof}  An equality-free formula is satisfiable iff it is satisfiable in an infinite domain. If any such  formula is satisfiable, it is so in either an infinite domain or a finite one,  and if the latter then use ~\cite[Lem.\ 2.24]{bridge} to find an infinite domain where it is satisfiable as well. Now, observe that for any equality-free formula $\phi$  satisfiable in an infinite model $\mathfrak{A}$, we have that $\vsim \phi$. This is because any model $\mathfrak{B}$ for the empty language (so that $\mathfrak{B}$ is simply a set) can be vacuously elementarily extended (since there is no equality in the language and no predicate, constant or function symbols) to a model $\mathfrak{B}'$, also in the empty language, that can be expanded to a model  $\mathfrak{B}''$ of $\phi$ (just take $\mathfrak{B}''$ to be a model of $\phi$ of some cardinality bigger than or equal to that of $\mathfrak{A}$ -- this model exists using ~\cite[Lem.\ 2.24]{bridge} conjoined with the fact that $\mathfrak{A}\models \phi$).
\end{proof}}

{
\begin{Pro}
[Fourth Reduction Case] \label{red4}For any sentence $\phi$  of $L$ not involving any other predicate than the equality symbol `$=$', $\vsim \phi$ iff $\vdash \phi$.
\end{Pro}
\begin{proof}  We start  by observing that   $\phi$ has models of all sizes only if $\vsim \phi$. Suppose first  that $\phi$ has models of all sizes. Then let $\mathfrak{A}$ be a model for the empty language, so that  $\mathfrak{A}$ may be identified with set. But then surely it can be expanded to a model of $\phi$ by the hypothesis that $\phi$ has models of all sizes. Now we claim that 
$\vsim \phi$ iff $\phi$ has models of all \emph{finite} sizes. If $\vsim \phi$ then any finite set can be expanded to a model of $\phi$ and hence $\phi$ has models of \emph{all} finite cardinalities. If the latter, on the other hand, given any finite set we can find a model of $\phi$ of the same cardinality as the original set and when seen as a model in the empty language it will be trivially isomorphic to the original set. Moreover, by the compactness theorem, any formula $\phi$ has models of all cardinalities if it has models of all finite cardinalities. (If $\phi$ has models of all finite cardinalities then the theory $T$ formed by $\phi$ and the set of sentences `there are at least $n$ elements' for every $n$, is finitely satisfiable and by compactness, satisfiable in an infinite model, so, using the upwards L\"owenheim-Skolem theorem, in models of any infinite cardinality.) Assume that $\phi$ has models of all finite sizes. Since the language only has equality this means that $\phi$ is true in all finite models, as they are just sets. Suppose for  contradiction that  $\nvdash \phi$, that is, $\neg \phi$ has a model, but then by  \cite[Thm. 1]{Ash} it must have a finite model, which by hypothesis must also be model of $\phi$. Thus $\phi$ has models of all finite sizes iff $\vdash \phi$.
\end{proof}}

{A simple consequence of Proposition \ref{red4} is that,  in contrast to the propositional case \cite{Makinson}, here we do not always get that $\emptyset\vsim \phi$  if $\phi$ is consistent. In particular, if $\phi$ is any pure equality sentence such that $\not \vdash \phi$ we will also get that $\emptyset \not\vsim \phi$.}

\begin{Pro}[Characterization in terms of consistency]\label{consis} Let $\Gamma$ be a set of sentences and $\phi$ a sentence from $L$.  Then the following are equivalent:
\begin{itemize}
\item[$(i)$]$\Gamma \vsim \phi$

\item[$(ii)$] $\phi$ is consistent with every set  $\Delta$ of sentences in  the language $L_\Gamma$ that is consistent with $\Gamma$.
\end{itemize}

\end{Pro}

\begin{proof} $(i) \implies (ii)$: Suppose $\Gamma \vsim \phi$ and $\Delta$  is  a set of sentences in  the language $L_\Gamma$ consistent with $\Gamma$, so that there is a model $\mathfrak{A} \models \Gamma \cup \Delta$ for the language $L_\Gamma$. But then there is a model   $\mathfrak{A}'$  for the same language such that $\mathfrak{A}\preceq_{L_\Gamma}
 \mathfrak{A}'$ (and thus  $\mathfrak{A}'\models \Delta$) and, furthermore,  $\mathfrak{A}'$ can be expanded to a model $\mathfrak{A}''$ for the language $L_{\Gamma,\phi}$ for which $\mathfrak{A}''\models \phi$. Consequently, $\phi$ is consistent with $ \Delta$.

$(ii) \implies (i)$: Assume that $\Gamma \not \vsim \phi$, that is, there  is  a model  $\mathfrak{A} $ for the language $L_\Gamma$ such that $\mathfrak{A} \models \Gamma$ and for any  $\mathfrak{A}'$ with $\mathfrak{A}\preceq_{L_\Gamma}
 \mathfrak{A}'$,  any expansion $\mathfrak{A}''$ of $\mathfrak{A}'$ to the language $L_{\Gamma,\phi}$  is such that $\mathfrak{A}'' \not \models \phi$, i.e. $\mathfrak{A}'' \models \neg\phi$. Take now $\Delta=\text{Th}_{L_\Gamma}(\mathfrak{A})$ (the complete first-order theory of $\mathfrak{A}$ in $L_\Gamma$). We show that $\phi$ is not consistent with $\Delta$ (clearly $\Delta$ is consistent with $\Gamma$ by construction). Suppose for a contradiction that $\Delta \cup \{\phi\}$ is consistent; then one can show using Robinson diagrams as before that there is a model $\mathfrak{A}'' \models \phi$ for the language $L_{\Gamma,\phi}$ such that $\mathfrak{A}\preceq_{L_\Gamma}
 (\mathfrak{A}''\upharpoonright L_\Gamma)$ (where $(\mathfrak{A}''\upharpoonright L_\Gamma)$ is the restriction of $\mathfrak{A}''$ to the language $L_\Gamma$). Choosing  $\mathfrak{A}'$ as $\mathfrak{A}''$ restricted to $L_\Gamma$, and noting that  $\mathfrak{A}''$ is an expansion of $\mathfrak{A}'$ to the language $L_{\gamma, \phi}$ we also have  $\mathfrak{A}''$ does not satisfy $\phi$, giving the desired contradiction.
 \end{proof}

The characterization in terms of consistency can then be refined as follows, with the same verifications as in \cite{Makinson}. {Observe that those verifications made use of Craig's interpolation \cite{Craig}  for classical consequence and hence needed the presence of a falsum or verum connective in the language (in the presence of equality these are definable). }

\begin{Pro}[Refinement] Let $\Gamma$ be a set of sentences and $\phi$ a sentence from $L$. Then the following are equivalent:
\begin{itemize}
\item[$(i)$]$\Gamma \vsim \phi$

\item[$(ii)$] $\Gamma \vdash \psi$  for every sentence $\psi$ of the language $L_\Gamma$ with $\phi \vdash \psi$.
\item[$(iii)$] $\Gamma \vdash \psi$  for every sentence $\psi$ of the language $L_\Gamma \cap L_\phi$ with $\phi \vdash \psi$.
\end{itemize}

\end{Pro}

\begin{Rmk}\label{use3}
\emph{The characterization in terms of consistency and its refinements all fail for the variation of Definition \ref{friend} where $\mathfrak{A}' = \mathfrak{A}$ (although the direction $(i) \implies (ii)$ in the characterization in terms of consistency still holds). This corresponding friendliness relation is denoted as $\vsim^{\text{{\bf R}}_1}_{\text{{\bf S}}_1}$ in the Appendix. If one works in a language $L^-$ without equality, it is possible to use the counterexample in \cite[p. 6]{Makinson} to see this (there $\Gamma$ is the set of all first-order consequences of $\forall x Px$ and $\phi$ is the sentence $\exists x \exists y (Rxy \wedge \neg Ryx)$). If on the other hand, we are working in full $L$, that example no longer does the trick because $\exists x \exists y (Rxy \wedge \neg Ryx)$ is not consistent with $\exists x \forall y(x=y)$ while $\forall x Px$  is, and we need to appeal to a more subtle one. The answer is given using non-resplendent structures. (For the theory of resplendent structures the reader can consult \cite{Barwise, Kossak, Hodges}.) A structure $\mathfrak{A}$ is \emph{resplendent} if all existential second-order formulas, that is, formulas of the form $\exists R \phi(R)$ where $\phi$ is a first-order formula, that are true in some elementary extensions of  $\mathfrak{A}$, are already true in  $\mathfrak{A}$.  We will borrow the following nice example from \cite{Kossak}: $\text{Th}(\mathds{Z}, +, \times)$, the complete theory of the structure   $(\mathds{Z}, +, \times)$, is not friendly to the sentence $$\phi(I) := \exists x I(x) \wedge \exists x \neg I(x) \wedge \forall x \forall y (x+1=y \rightarrow (I(x) \leftrightarrow I(y))),$$
as the structure $(\mathds{Z}, +, \times)$ does not have any proper subset of its universe closed under successors and predecessors. On the other hand, every set of sentences of the language of $\text{Th}(\mathds{Z}, +, \times)$ consistent with $\text{Th}(\mathds{Z}, +, \times)$ is also consistent with $\phi(I)$ as one can see by taking  any proper elementary extension of $(\mathds{Z}, +, \times)$. } 
\end{Rmk}

\begin{Rmk}
\emph{However, if one  restricts attention to resplendent models only, the version of $\vsim$ introduced in Definition \ref{friend} and $\vsim^{\text{{\bf R}}_1}_{\text{{\bf S}}_1}$, which simply states the existence of an expansion, become equivalent.  Clearly, if $\Gamma \vsim^{\text{{\bf R}}_1}_{\text{{\bf S}}_1} \phi$ in the second sense, then $\Gamma \vsim \phi$ in the sense of Definition \ref{friend} (any model is trivially an elementary extension of itself). On the other hand if $\Gamma \vsim \phi$ in the sense of Definition \ref{friend} then, for any resplendent model of $\Gamma$, if the elementary extension that witnesses this fact can be expanded to a model of $\phi$, then,  by resplendence, the original model already can be expanded in such a way.  A useful fact of resplendent structures is that any model can be elementarily embedded in a resplendent structure of the same cardinality as the original model \cite[Claim (i), p. 534]{Barwise}. Now let us show that if $\Gamma \not\vsim \phi$ when considering all models then  $\Gamma \not\vsim \phi$ when restricting to resplendent models. Assume the sense of Definition \ref{friend} again that $\Gamma \not\vsim \phi$ and take a model $\mathfrak{A}\models \Gamma$ witnessing this fact, so no elementary extension of it can be expanded to a model of $\phi$. But then there must be a resplendent elementary extension   $\mathfrak{B}$ of $\mathfrak{A}$ such that the same occurs (for otherwise, $\mathfrak{B}$  itself could already be expanded to a model of $\phi$ contradicting our choice of $\mathfrak{A}$). Hence, if the $\vsim$ relation is defined on resplendent models alone, one can make $\mathfrak{A} = \mathfrak{A}'$ in Definition \ref{friend} without loss of generality. } 
\end{Rmk}

{\begin{Rmk}
\emph{
$\vsim^{\text{{\bf R}}_1}_{\text{{\bf S}}_1}$ is roughly the complement of a relation that was articulated  by de Bouvère  \cite{Bou} in
the context of the theory of definition, to serve as a tool for showing the non-definability of a
predicate $P$ in a first-order theory. To bring out the connection, we state a certain instance of de Bouvère's Lemma 1 in
our notation. Let $\Delta$ be a first-order theory and $P$ a predicate of its language $L=L_\Delta$. Write $L\setminus P$ for
that language less $P$.  Then a sufficient condition for $P$ to be undefinable in $Cn_{L\setminus P}(\Delta)$ (where this is the set of consequences of $\Delta$ in the language $L\setminus P$) is that there is a
model of $Cn_{L\setminus P}(\Delta)$ for the language  $L\setminus P$  that cannot be extended to a model of $\Delta$ by giving an interpretation of $P$. In
the case that $\Delta$ is a singleton $\{\phi\}$ the condition
becomes: there is a model of $Cn_{L\setminus P}(\phi)$  that cannot be extended to a model of $\phi$ by giving an
interpretation of $P$; in other words, it requires that $Cn_{L\setminus P}(\phi)\not \vsim^{\text{{\bf R}}_1}_{\text{{\bf S}}_1} \phi$. In this way, a special case
of de Bouvère's condition (the one where $\Delta$ is a singleton) requires that a certain pair is not in the relation  $\ \vsim^{\text{{\bf R}}_1}_{\text{{\bf S}}_1}$.}
\end{Rmk}}

Next we mention some more properties that carry over from the propositional case:

\begin{Pro} Let $\Gamma, \Delta$ be sets of sentences and $\phi, \psi$  sentences from $L$. Then 

\begin{itemize}
\item[(i)] {\bf (Right weakening)}\emph{:} $\Gamma \vsim \phi \vdash \psi$ only if $\Gamma \vsim \psi$.

\item[(ii)] {\bf (Singleton cumulative transitivity)}\emph{:} whenever $\Gamma \vsim \phi$ and $\Gamma, \phi \vsim \psi$  we have that  $\Gamma \vsim \psi$.
\item[(iii)] {\bf (Local left strengthening)}\emph{:} suppose that $L_\Delta \subseteq L_\Gamma$. Then  $ \Delta \vdash \Gamma \vsim \phi$  only if $\Delta \vsim \phi$.
\item[(iv)] {\bf (Local left equivalence)}\emph{:} suppose that $L_\Delta \subseteq L_\Gamma$. Then  $  \Gamma \vsim \phi$ and $\Gamma \dashv \vdash \Delta$  only if $\Delta \vsim \phi$.

\item[(v)] {\bf (Local monotony)}\emph{:} suppose that $L_\Delta \subseteq L_\Gamma$, $  \Gamma \vsim \phi$ and $\Gamma \subseteq \Delta$  only if $\Delta \vsim \phi$.
\item[(vi)] {\bf (Local disjunction in the premisses)}\emph{:} Suppose $L_\psi \subseteq L_{\Gamma,\phi} $ and  $L_\phi \subseteq L_{\Gamma,\psi} $.
Then
 $  \Gamma, \phi \vsim \chi$ 
 and  $  \Gamma, \psi \vsim \chi$  together imply $  \Gamma, \phi \vee \psi \vsim \chi$.
\item[(vii)] {\bf (Proof by exhaustion)}:  $  \Gamma, \phi \vsim \chi$ 
 and  $  \Gamma, \neg \phi \vsim \chi$  together imply $  \Gamma\vsim \chi$.
\end{itemize}

\end{Pro}

\begin{proof} (i) follows from the definition, while (iv), (v) are immediate consequences of (iii) and (vii) is immediate from (vi). The verifications for the remainder are essentially the same as for the propositional case, as set out in \cite{Makinson}, but we run through them for the reader's convenience. Suppose that $\Gamma \vsim \phi$ and $\Gamma, \phi \vsim \psi$. Take any model $\mathfrak{A} \models \Gamma$ for the language $L_\Gamma$. By  hypothesis, 
there is a model  $\mathfrak{A}'$ for the same language such that both $\mathfrak{A}\preceq_{L_\Gamma}
 \mathfrak{A}'$ and $\mathfrak{A}'$ can be expanded to a model $\mathfrak{A}''$ for the language $L_{\Gamma,\phi}$ for which $\mathfrak{A}''\models \phi$. But now, by hypothesis again, one may find $\mathfrak{A}'''$ with $\mathfrak{A}''\preceq_{L_{\Gamma,\phi}}
 \mathfrak{A}'''$ (and hence $\mathfrak{A}\preceq_{L_\Gamma}
 (\mathfrak{A}'''\upharpoonright L_\Gamma$)) such that $ \mathfrak{A}'''$ (and hence $(\mathfrak{A}'''\upharpoonright L_\Gamma$)) can be expanded to a model $ \mathfrak{A}'''' \models \psi$.
 
(iii): Assume  that  $L_\Delta \subseteq L_\Gamma$ and $ \Delta \vdash \Gamma \vsim \phi$. Take any model $\mathfrak{A} \models \Delta$ for the language $L_\Delta$. By hypothesis, since $ \Delta \vdash \Gamma$, any expansion $\mathfrak{A}' $ of $\mathfrak{A}$ to $L_\Gamma$ must be such that  $\mathfrak{A}' \models \Gamma$. But then, since $\Gamma \vsim \phi$, we have that there is a model  $\mathfrak{A}''$ for the same language such that $\mathfrak{A}'\preceq_{L_\Gamma}
 \mathfrak{A}''$ and, furthermore,  $\mathfrak{A}''$ can be expanded to a model $\mathfrak{A}'''$ for the language $L_{\Gamma,\phi}$ for which $\mathfrak{A}'''\models \phi$. So $\mathfrak{A} = (\mathfrak{A}'  \upharpoonright L_\Delta) \preceq_{L_\Delta}
( \mathfrak{A}''\upharpoonright L_\Delta)$, and we have that $\Delta \vsim \phi$ as desired.

(vi): Suppose that $  \Gamma, \phi \vee \psi \not \vsim \chi$, so there is a model  $\mathfrak{A}$ for the language $L_{\Gamma,\phi,\psi}$ such that  $\mathfrak{A}\models \Gamma $ and $\mathfrak{A}\models \phi \vee \psi $ such that there is no  model  $\mathfrak{A}'$ for $L_{\Gamma,\phi,\psi}$ such that both $\mathfrak{A}\preceq_{L_{\Gamma,\phi,\psi}}
 \mathfrak{A}'$ and  $\mathfrak{A}'$ can be expanded  to a model $\mathfrak{A}''$ for the language $L_{\Gamma,\phi,\psi,\chi}$ for which $\mathfrak{A}''\models \chi$. By the hypotheses, $L_{\Gamma,\phi,\psi} = L_{\Gamma,\phi} = L_{\Gamma, \psi}$, so either $\mathfrak{A}\models \phi$ or $ \mathfrak{A}\models \psi $, and thus either  $  \Gamma, \phi  \not\vsim \chi$ or  $  \Gamma, \psi  \not\vsim \chi$ as desired.
\end{proof}

\section{First-order discontinuities}\label{dis}

We know from \cite{Makinson} that compactness holds in quite a strong form in the propositional context. But it fails in the first-order case because we are requiring, in Definition \ref{friend} of first-order friendliness, that the model $\mathfrak{A}''$ has the same domain as the model $\mathfrak{A}'$ (which in turn is elementarily equivalent to  $\mathfrak{A}$).  This permits the construction of a counter-example to compactness as follows.

\begin{Pro}[Compactness fails for $\vsim $]\label{comp} There is a set $\Gamma$  of sentences such that $\Gamma \vsim \phi$ but for no finite set $\Gamma_0\subseteq \Gamma$ do we have that $\Gamma_0 \vsim \phi$.
\end{Pro}
\begin{proof}
Consider the first-order theory and sentence $$\Gamma := \{\text{`there are at least $n$ elements'}   \ | \ n > 0  \}$$  
 $$ \phi := \text{`the relation $R$ is an injective but not surjective total mapping on the domain'}.$$ It is easy to see that $\Gamma \vsim \phi$: any model $\mathfrak{A}\models \Gamma$ has to be infinite, so it can clearly be expanded to a model of $\phi$ (which simply restates Dedekind's definition of infinity). On the other hand, if $\Gamma_0\subseteq \Gamma$ is finite, it would surely have a sufficiently large finite model $\mathfrak{A}$. Suppose for a contradiction  that $\Gamma_0 \vsim \phi$, so there is a model  $\mathfrak{A}'$ for the  language $L_{\Gamma_0}$ such that $\mathfrak{A}\preceq_{L_{\Gamma_0}}
 \mathfrak{A}'$ (observe that since $\mathfrak{A}$ is finite and we are assuming that the language contains equality,  $\mathfrak{A}'$ must be the same size)  and, furthermore,  $\mathfrak{A}'$ can be expanded to a model $\mathfrak{A}''$ for the language $L_{\Gamma_0,\phi}$ for which $\mathfrak{A}''\models \phi$, but this is a contradiction as it would imply that  $\mathfrak{A}'$  is infinite.
\end{proof}

\begin{Rmk}\label{comp2}\emph{The argument for Proposition \ref{comp}  works equally well for the variant of Definition \ref{friend} where $\mathfrak{A} = \mathfrak{A}'$.
On the other hand it does not work when the language lacks equality; indeed, in that case we have a quite trivial (and rather uninteresting) form of compactness. Let $\phi$ be a sentence in a language $L^-_\phi$ and $\Gamma$  a set of sentences in a language $L^-_\Gamma$. If $\Gamma \vsim \phi$, either $\phi$ is inconsistent or not. If the first, $\Gamma$ itself must be inconsistent, so by compactness for $\vdash$, there is a finite subset $\Gamma_0\subseteq \Gamma$ that is inconsistent. Trivially, $\Gamma_0 \vsim \phi$ in this case. If, on the other hand, $\phi$ is consistent we have that $\emptyset \vsim \phi$. To see this take any model $\mathfrak{A}$ of $\emptyset$, that is, any set $A$, since $\emptyset$ is a theory in the empty language, and consider a model $\mathfrak{B}\models \phi$. Either $|B| \leq |A|$ or $|A| \leq |B|$.   If the former, that is $|B| \leq |A|$, by a 
known result of first-order logic without equality~\cite[Lem.\ 2.24]{bridge} or \cite[Ch. IV, \S 1]{ack} (stated more generally  in \cite[Lemma 3]{Badia}), there is a model $\mathfrak{C} \models\phi$ such that $|B| \leq |C| =|A|$. Hence, letting the set $C$ be considered as a model $\mathfrak{A}'$ for $\emptyset$ we have that it is a trivial (since the language is empty) elementary extension of $\mathfrak{A}$, so we are done. If, on the other hand, we have the latter possibility, then there is an injection from $A$ into $B$, and hence letting $\mathfrak{A}'= A$ and  $\mathfrak{C}=\mathfrak{B}$  we get the models  needed to witness $\emptyset \vsim \phi$.  Incidentally, this is a simple and naturally arising example of proof by cases using an undecidable criterion (whether $\phi$ is consistent or not) for the two cases (cf.\cite{kle2}).}
\end{Rmk}

We can also construct a counter-example to interpolation.

\begin{Pro}[Interpolation fails for the relation $\vsim $] \label{inter}There is a set of sentences $\Gamma$ and sentence $\phi$ such that $\Gamma \vsim \phi$ but there is no interpolant $\psi$  in the language $L_\Gamma \cap L_\phi$ such that  $\Gamma \vsim \psi$ and  $\psi \vsim \phi$.
\end{Pro}

\begin{proof}
Consider the theory $\Gamma$ and sentence $\phi$ in the proof of Proposition \ref{comp}. We know that $\Gamma \vsim \phi$ and $\Gamma$ is consistent. However, if there is an interpolant $\psi$ which uses equality in the language $L_\Gamma \cap L_\phi$ (which is just the pure language of equality) such that  $\Gamma \vsim \psi$ and  $\psi \vsim \phi$, this means that $\psi$ only has infinite models.  To see this, suppose that $\psi \vsim \phi$ and $\mathfrak{A} \models \psi$ where  $\mathfrak{A}$ is a model for the pure language of equality. It follows that there is a model  $\mathfrak{A}'$ for the same language such that $\mathfrak{A}\preceq_{L_\psi}
 \mathfrak{A}'$  and that, furthermore,  $\mathfrak{A}'$ can be expanded to a model $\mathfrak{A}''$ for the language $L_{\psi,\phi}$ for which $\mathfrak{A}''\models \phi$, and hence $\mathfrak{A}'$ has to be infinite (and then $\mathfrak{A}$ is infinite too as the pure-equality sentences `there are at least $n$ elements' are all true in $\mathfrak{A}'$ and thus in  $\mathfrak{A}$ since $\mathfrak{A}\preceq_{L_\psi}
 \mathfrak{A}'$).  Since $\Gamma $ is consistent and $\Gamma \vsim \psi$, then  $\psi$  clearly has a model, and hence we obtain a contradiction with \cite[Thm. 1]{Ash} which implies that $\psi$ has a finite model since it is a sentence in the pure language of equality.
\end{proof}

\begin{Rmk}\label{inter3}\emph{This counterexample to interpolation also works for the variant of Definition \ref{friend} where $\mathfrak{A} = \mathfrak{A}'$. We note that if the language lacks equality but has a zero-ary $\bot$ connective, then the counter-example in Proposition \ref{inter} is no longer available and in fact interpolation then holds in a rather trivial way, as can be shown by essentially the same kind of argument as in Remark \ref{comp2} (putting $\bot$ in place of $\Gamma_0$ in the case that $\Gamma$ is inconsistent, and putting  $\top$ in place of the empty set when $\phi$ is consistent).
}
\end{Rmk}

We have seen  the failure of  interpolation, so what about  the Beth definability theorem? As it is well-known,  interpolation gives Beth definability in first-order logic (the other direction is false \cite{Jensen} in general when considering abstract logics in the sense of Lindstr\"om \cite{lin}). Thus even though we have disproved the former it is still reasonable to wonder about the latter$-$and it comes out true.

\begin{Pro}[Beth definability property for $\vsim $] \label{beth} Let $P$ be an $n$-ary predicate symbol not in language $L$, $\Gamma$ a set of sentences where $L_\Gamma = L \cup \{P\}$  and $\mathfrak{A}$ a  model for $L$. Then the following are equivalent:
\begin{itemize}
\item[$(i)$] If $(\mathfrak{A}, B) \models \Gamma$ and $(\mathfrak{A}, C) \models \Gamma$ (where $(\mathfrak{A}, B), (\mathfrak{A}, C)$ are the expansions of $\mathfrak{A}$ obtained by interpreting the predicate $P$ as the $n$-ary relations $B$ and $C$, respectively),  then $B=C$.
\item[$(ii)$] There is a formula $\psi (\overline{x})$ of $L$ such that
$$ \Gamma \vsim \forall \overline{x} (\psi (\overline{x}) \leftrightarrow P (\overline{x})).$$
\item[$(iii)$]There is a formula $\psi (\overline{x})$ of $L$ such that
$$ \Gamma \vdash \forall \overline{x} (\psi (\overline{x}) \leftrightarrow P (\overline{x})).$$
\end{itemize}
\end{Pro}

\begin{proof}
$(i) \implies (ii)$: Suppose (i), Beth's theorem for classical logic gives us (ii) with $\vdash$ in place of $\vsim$, and we are done by Proposition \ref{supra} .

$(ii) \implies (i)$: Assume that 
there is a formula $\psi (\overline{x})$ of $L$ such that
$$ \Gamma \vsim  \phi$$ 
where $\phi$ is $\forall \overline{x} (\psi (\overline{x}) \leftrightarrow P (\overline{x}))$ and where $L_\Gamma = L \cup \{P\}$ with $P \notin L$.
Suppose that $(\mathfrak{A}, B) \models \Gamma$ and $(\mathfrak{A}, C) \models \Gamma$. We need to show that $B=C$. Since $ \Gamma \vsim  \phi$ 
 there is a model  $(\mathfrak{A}, B)'$ for the language $L_\Gamma$  such that $(\mathfrak{A}, B)\preceq_{L_\Gamma}
 (\mathfrak{A}, B)'$ which can be expanded to a model $(\mathfrak{A}, B)''$ for the language $L_{\Gamma,\phi}$  such that $(\mathfrak{A}, B)''\models \phi$. But $L_{\Gamma, \phi} = L_\Gamma=L\cup \{P\}$ so  $(\mathfrak{A}, B)' = (\mathfrak{A}, B)''$, so $(\mathfrak{A}, B)'\models \phi$. Since $(\mathfrak{A}, B)\preceq_{L_\Gamma}
 (\mathfrak{A}, B)'$ this gives us $(\mathfrak{A}, B)\models \phi$. By similar reasoning, $(\mathfrak{A}, C)\models \phi$. Since $\phi$ is $\forall \overline{x} (\psi (\overline{x}) \leftrightarrow P (\overline{x}))$, we thus have $B=C$.

 Finally, $(iii)$ is equivalent to $(i)$ by the Beth definability theorem.

\end{proof}
{
\begin{Rmk} Observe that the equivalence between $(ii)$ and $(iii)$ is also just an immediate consequence of Proposition \ref{red1}.
\end{Rmk}}

\begin{Rmk}\label{count}
\emph{The argument for the Beth definability property works as well for the variant of Definition \ref{friend} where $\mathfrak{A} = \mathfrak{A}'$.
On the other hand, for both definitions, if in the version of Beth definability that we have given for $\vsim$ we weaken the requirement to  $L_\Gamma \subseteq  L \cup \{P\}$,  the implication $(ii) \implies (i)$ fails. This is because if $P$ does not appear in $\Gamma$ then surely, for any $\psi (\overline{x}) $ of  $ L$, any model of $\Gamma$ can be expanded to one where $\forall \overline{x}  (\psi (\overline{x}) \leftrightarrow P (\overline{x}))$ is true. If $ \Gamma $ has a model, $\mathfrak{A}$, one can expand it to one for the language containing $P$ by interpreting $P$ as the elements of $\mathfrak{A}$ that satisfy $\psi$, or, alternatively,  the elements of $\mathfrak{A}$ that satisfy $\neg \psi$. This counter-example to $(ii) \implies (i)$ does not affect the classical $(iii) \implies (i)$ because for any $\Gamma$ with $\Gamma \cup \{\psi\}$ included in $L$ and $P$ not in $L$ if
   $\Gamma \vdash \forall \overline{x} (\psi (\overline{x}) \leftrightarrow P (\overline{x}))$, any model of $\Gamma$ for the language including $P$ will have to give $P$ the interpretation of the set of all the elements of the domain that satisfy $\psi$. {In this case given that $\Gamma$  and $\psi$ are both in $L$ while $P$ is not, if  $\Gamma \vdash \forall \overline{x} (\psi (\overline{x}) \leftrightarrow P (\overline{x}))$ then  $P$ must be interpreted as the empty set.}}
\end{Rmk}

\section{Non-axiomatizability}
Even though Proposition \ref{comp} implies that there can be no \emph{strongly complete} (and sound) finitary axiomatization of $\vsim$ (as the existence of such an axiomatization would imply compactness for the consequence relation), it still makes sense  to ask whether a \emph{weakly complete} finitary axiomatization exists. It turns out that for predicate calculus this is not possible either. {In what follows, as in the usual predicate calculus, we always assume that we have as rich a vocabulary as necessary, that is, we have a countable list of predicate and function symbols of whatever arity we might wish.}

\begin{Pro}\label{nonax} The set of formulas $\phi$ in predicate calculus without equality for which  $\vsim \phi$, is not recursively enumerable.

\end{Pro}

\begin{proof} 
First recall that the set of validities of predicate calculus without equality is not recursive by the Church-Turing theorem \cite[Thm. 54 of \S 76]{kle2}.  Thus the set of  formulas without equality that are satisfiable  cannot be recursively enumerable since the set of validities is recursively enumerable. {We then use Proposition \ref{red3} to get the result.} \end{proof}

{
\begin{Cor} \label{nonax2} The set of formulas $\phi$ in predicate calculus with equality  for which  $\vsim \phi$, is not recursively enumerable.
\end{Cor}
\begin{proof} Clearly, any effective enumeration of the formulae $\phi$ with equality such that $\vsim \phi$, can be transformed into an effective enumeration of those among them that lack equality, contrary to Proposition \ref{nonax}.
\end{proof}}

{\begin{Rmk}\label{nonax3}\emph{
It is natural to ask what happens if we have an empty vocabulary and the pure language of equality. In this case, we have recursive enumerability thanks to Proposition \ref{red4}.}
\end{Rmk}}

\section*{Open questions}

There are many directions for further research. Evidently, a first concern is to complete Table 1 in the Appendix. One could also look for other relations between models that make sense as specifications of {\bf R} there; back-and-forth equivalence is mentioned in the Appendix as a possibility. One might consider in a more systematic way the context of equality-free logic, where the relation of weak isomorphism used in \cite{Badia} could be a tool of interest. Finally, one might go beyond the context of classical logic to explore intuitionistic or modal contexts, in which further specifications of {\bf R} are suggested by bisimulation and similar notions. 
\section*{Acknowledgements}
 {We are deeply indebted to an anonymous reviewer who provided very detailed and helpful comments identifying a number of problems in an earlier version of the article.} Badia is supported by  the Australian Research Council grant DE220100544.

\section*{Appendix}\label{gen}

A fairly general format for variants of Definition \ref{friend} in first-order contexts (with equality)
is as follows:
\begin{itemize}
\item[] Say that a set $\Gamma$ of sentences is \emph{friendly$^\text{{\bf R}}_\text{{\bf S}}$}  (in symbols $\vsim^\text{{\bf R}}_\text{{\bf S}}$) to a formula  $\phi$ iff for every model $\mathfrak{A}$ for the language
$L_\Gamma$, if $\mathfrak{A}$ satisfies  $\Gamma$ then there are models $\mathfrak{A}'$, $\mathfrak{A}''$ for languages $L_\Gamma$, $L_{\Gamma, \phi}$, respectively with $(\mathfrak{A}, \mathfrak{A}') \in \text{{\bf R}}$, $(\mathfrak{A}', \mathfrak{A}'') \in \text{{\bf S}}$
 and $\mathfrak{A}''$ satisfies both $\Gamma$ and $\phi$.
 \end{itemize}
Various definitions emerge from different ways of specifying the relations {\bf R}, {\bf S}. There are at least four
interesting ways of filling in for {\bf R}, and another three for {\bf S}. Our guiding idea is that  {\bf R} is some kind of relation of being `... is essentially the same as ...', and {\bf S} one of  `... can be enlarged to ...'. We choose {\bf R} to cover relations at least as strong as elementary equivalence bearing in mind the remark of Angus Macintyre that in the model theory of  first-order logic `elementary equivalence is the most fundamental concept [...] It is the analogue of isomorphism in algebra' \cite[p. 140]{Mc}. One example which we have not mentioned so far is back-and-forth equivalence in terms of Ehrenfeucht-Fra\"iss\'e games \cite[Corollary 3.3.3]{Hodges} but for simplicity we will leave it out of our analysis, which is by no means complete.
Here are some options for {\bf R}: $(\mathfrak{A}, \mathfrak{A}') \in \text{{\bf R}}$ iff

\begin{itemize}
\item[{\bf R}$_1$:]  $\mathfrak{A} = \mathfrak{A}'$

\item[{\bf R}$_2$:] $\mathfrak{A} $ is isomorphic to $\mathfrak{A}'$

\item[{\bf R}$_3$:]   $\mathfrak{A} $ can be elementarily embedded into $\mathfrak{A}'$\item[{\bf R}$_4$:]  iff $\mathfrak{A} $ is elementarily equivalent to $\mathfrak{A}'$

\end{itemize}
Clearly each option $\text{{\bf R}}_i$ implies $\text{{\bf R}}_{i+1}$. Our Definition \ref{friend} uses $\text{{\bf R}}_3$.

As  regards ways of filling {\bf S}, the situation is more subtle. One could think of {\bf S} as a being described  by a triple $\langle  D_i, C, P_j \rangle$ where $i,j \in \{1,2\}$ where:
\begin{itemize}
\item[] $D_1$ says that $A' = A''$,
\item[] $D_2$ says that $A' \subseteq A''$,
\item[] $C$ says that for every constant symbol occurring in $\Gamma$, its value in $\mathfrak{A}''$ is the same as its value in $\mathfrak{A}'$,
\item[] $P_1$ says that for every predicate symbol occurring in $\Gamma$, its value in $\mathfrak{A}''$  is the same as its value in $\mathfrak{A}'$,
\item[] $P_2$ says that for every predicate symbol occurring in $\Gamma$, its value in $\mathfrak{A}''$ intersected with the appropriate
power of $A'$,  is the same as its value in $\mathfrak{A}'$.
\end{itemize}
Clearly, $D_1$ implies $D_2$ and when $D_1$ holds then $P_1$ is equivalent to $P_2$. So we get three ways of filling in
for {\bf S}: $\langle  D_1, C, P_1 \rangle$  is equivalent to $\langle  D_1, C, P_2 \rangle$  which in turn implies  $\langle  D_2, C, P_1 \rangle$ which implies  $\langle  D_2, C, P_2 \rangle$. Hence, we are left with three possibilities: 
\begin{itemize}
\item[] 
{\bf S}$_1$ given by $\langle  D_1, C, P_1 \rangle$,
\item[]  {\bf S}$_2$ given by $\langle  D_2, C, P_1 \rangle$, and
 \item[]  {\bf S}$_3$ given by $\langle  D_2, C, P_2 \rangle$. 
  \end{itemize}Definition \ref{friend} uses specification {\bf S}$_1$, the strongest in this list, so the relation defined is $\vsim^{\text{{\bf R}}_3}_{\text{{\bf S}}_1}$ while a variant of it that is mentioned from time to time in the main text is $\vsim^{\text{{\bf R}}_1}_{\text{{\bf S}}_1}$. 

\begin{Rmk}\label{Srel}\emph{
Before we continue, let us reflect for a minute on all the three options {\bf S}$_i$. Clearly, {\bf S}$_1$ tells us that $\mathfrak{A}' = (\mathfrak{A}'' \upharpoonright L_\Gamma)$. In contrast,
 {\bf S}$_2$ tells us that $\mathfrak{A}' \subseteq (\mathfrak{A}'' \upharpoonright L_\Gamma)$, that is, $\mathfrak{A}'$ is a submodel of $(\mathfrak{A}'' \upharpoonright L_\Gamma)$ and, furthermore, it is a \emph{pure submodel}\cite[p. 528]{Hodges} in the equality-free language $L_\Gamma^-$ in the sense that any \emph{primitive positive existential} formula of $L_\Gamma^-$ (i.e. formula of the form $\exists \overline{y} \phi(\overline{y})$ where $\phi(\overline{y})$ is a conjunction of atomic formulas, excluding identities) that is satisfied in $(\mathfrak{A}'' \upharpoonright L_\Gamma)$ by a sequence of elements from $\mathfrak{A}'$  must be satisfied in $\mathfrak{A}'$ (equivalently, all negations of existential primitive positive formulas satisfied in  $\mathfrak{A}'$  are satisfied in $(\mathfrak{A}'' \upharpoonright L_\Gamma)$). On the other hand, {\bf S}$_3$ only tells us that $\mathfrak{A}' \subseteq (\mathfrak{A}'' \upharpoonright L_\Gamma)$.}\end{Rmk}

\begin{Pro}\label{genR} Let {\bf R}$_i$  ($i\in \{1,2,3, 4\}$) and {\bf S}$_j$ ($j\in \{1,2,3\}$) be as above. Then,

\begin{itemize}
\item[$(i)$]  $\vsim^{\text{{\bf R}}_1}_{\text{{\bf S}}_j} \ = \ \vsim^{\text{{\bf R}}_2}_{\text{{\bf S}}_j} \ \subseteq \ \vsim^{\text{{\bf R}}_3}_{\text{{\bf S}}_j} \ \subseteq \ \vsim^{\text{{\bf R}}_4}_{\text{{\bf S}}_j} $ for $j \in \{1,2,3\}$, $\vsim^{\text{{\bf R}}_3}_{\text{{\bf S}}_1} \ \supseteq \ \vsim^{\text{{\bf R}}_4}_{\text{{\bf S}}_1} $, and
\item[$(ii)$]  $\vsim^{\text{{\bf R}}_i}_{\text{{\bf S}}_1} \ \subseteq \ \vsim^{\text{{\bf R}}_i}_{\text{{\bf S}}_2} \ \subseteq \ \vsim^{\text{{\bf R}}_i}_{\text{{\bf S}}_3}$ for $i \in \{1,2,3,4\}$.
\end{itemize}
\end{Pro}

\begin{proof}
(i): Clearly, all left-in-right inclusions hold since each $\text{{\bf R}}_i \subseteq \text{{\bf R}}_{i+1}$. The converse  inclusion $  \vsim^{\text{{\bf R}}_2}_{\text{{\bf S}}_j} \ \subseteq \ \vsim^{\text{{\bf R}}_1}_{\text{{\bf S}}_j}$ is easy to see  because we can isomorphically copy  $\mathfrak{A}''$ as an expansion of $\mathfrak{A}$ itself.
Now we wish to show that $\vsim^{\text{{\bf R}}_4}_{\text{{\bf S}}_1} \ \subseteq \ \vsim^{\text{{\bf R}}_3}_{\text{{\bf S}}_1} $. Suppose  that $\Gamma \vsim^{\text{{\bf R}}_4}_{\text{{\bf S}}_1} \phi$; we want to show that $\Gamma \vsim^{\text{{\bf R}}_3}_{\text{{\bf S}}_1} \phi$ (which is just what we have called  $\Gamma \vsim \phi$ until now).  
 Let then $\mathfrak{A} \models \Gamma$ be an arbitrary model for $L_\Gamma$, so by the assumption $\Gamma \vsim^{\text{{\bf R}}_4}_{\text{{\bf S}}_1} \phi$,
there is a model  $\mathfrak{A}'\equiv_{L_\Gamma}
 \mathfrak{A}$ which can be expanded to a model $\mathfrak{A}''$ for the language $L_{\Gamma, \phi}$ with $\mathfrak{A}''\models \phi$. In order to establish that $\Gamma \vsim^{\text{{\bf R}}_3}_{\text{{\bf S}}_1} \phi$ it suffices  to  show that $\text{eldiag}(\mathfrak{A}) \cup \{\phi\}$ is consistent. Suppose otherwise, that is,  for some finite $T \subseteq \text{eldiag}(\mathfrak{A})$,  the theory $T \cup \{\phi\}$ is not consistent. Then $\phi \vdash \neg \bigwedge T$, but then since the diagram constants in $T$ are all new to $L_{\Gamma, \phi}$, $\phi \vdash \forall \overline{x}\neg \bigwedge T^*$ by \cite[Lemma 2.3.2]{Hodges} where $T^*$ is the result of replacing in every sentence in $T$ the new constants by new variables. Then $\mathfrak{A}''\models  \forall \overline{x}\neg \bigwedge T^*$ so $\mathfrak{A}'\models  \forall \overline{x}\neg \bigwedge T^*$ and since  $\mathfrak{A}' \equiv_{L_\Gamma} \mathfrak{A}$, we have that $\mathfrak{A}\models  \forall \overline{x}\neg \bigwedge T^*$. {On the
other hand, since  $(\mathfrak{A}, \overline{a}) \models T \subseteq \text{eldiag}(\mathfrak{A})$  (where $\overline{a}$  is a list of fresh names
for the elements of $\mathfrak{A}$), we have $\mathfrak{A}\models  \exists \overline{x} \bigwedge T^*$, giving us a contradiction}
 
% But since $\mathfrak{A}''$ is an expansion of  $\mathfrak{A}'$, we have $\mathfrak{A}'$ satisfies  $ \exists \overline{x} \bigwedge T^*$ and so, since  $\mathfrak{A}'$ is elementary equivalent to $\mathfrak{A}$ over the language $L_\Gamma$, we have $\mathfrak{A}\models  \exists \overline{x} \bigwedge T^*$, giving us a contradiction.
%
 
 (ii): This is immediate by the observations in  Remark \ref{Srel}  on the connections between {\bf S}$_1$, {\bf S}$_2$, and {\bf S}$_3$.\end{proof}

Thus it would be possible to reformulate Definition \ref{friend} equivalently with the attractively symmetric {\bf R}$_4$ in place of the non-symmetric {\bf R}$_3$. The reason why we have not followed that organization of the material is that our proofs using Robinson diagrams (for Propositions \ref{red2} and \ref{consis}) are more easily carried out in terms of {\bf R}$_3$ than {\bf R}$_4$.

{\begin{Rmk}\label{use1}\emph{Supraclassicality (Proposition \ref{supra}) trivially holds for all the inference relations just defined because we can always take $\mathfrak{A}=\mathfrak{A}'=\mathfrak{A}''$. Furthermore, for $\vsim^{\text{{\bf R}}_1}_{\text{{\bf S}}_1} \ = \ \vsim^{\text{{\bf R}}_2}_{\text{{\bf S}}_1}$ we get both reduction cases too. For the First Reduction Case we simply use Proposition \ref{genR} (i) and Proposition \ref{red1} itself. For the Second Reduction Case, we need to show $\Gamma \vsim^{\text{{\bf R}}_1}_{\text{{\bf S}}_1} \phi$ iff $\Gamma \not \vdash \neg \phi$.  If $\Gamma \vsim^{\text{{\bf R}}_1}_{\text{{\bf S}}_1} \phi$, then by Proposition \ref{genR} (i) and  Proposition \ref{red2},  $\Gamma \not \vdash \neg \phi$. The argument for the converse is a simplified version of that in the proof of Proposition \ref{red2} (the model $\mathfrak{B}$ that we are given in that reasoning is already all we need).}
\end{Rmk}
}

{\begin{Rmk}\emph{An inspection of the proof of Proposition \ref{nonax} on non-axiomatizability, reveals that it also holds for $\vsim^{\text{{\bf R}}_3}_{\text{{\bf S}}_2}, \vsim^{\text{{\bf R}}_3}_{\text{{\bf S}}_3}$, $ \vsim^{\text{{\bf R}}_4}_{\text{{\bf S}}_1}$, $ \vsim^{\text{{\bf R}}_4}_{\text{{\bf S}}_2}$, and $ \vsim^{\text{{\bf R}}_4}_{\text{{\bf S}}_3}$. This is because in each of these cases, a formula following from no premises implies satisfiability and  at the point of the proof of Proposition \ref{nonax} where we show that satisfiability of $\phi$ implies $\vsim^{\text{{\bf R}}_3}_{\text{{\bf S}}_1} \phi$ we may also conclude, using, Proposition \ref{genR}, that this is so for $\vsim^{\text{{\bf R}}_3}_{\text{{\bf S}}_2}, \vsim^{\text{{\bf R}}_3}_{\text{{\bf S}}_3}$, $ \vsim^{\text{{\bf R}}_4}_{\text{{\bf S}}_1}$, $ \vsim^{\text{{\bf R}}_4}_{\text{{\bf S}}_2}$, and $ \vsim^{\text{{\bf R}}_4}_{\text{{\bf S}}_3}$.}
\end{Rmk}
}

\begin{Rmk}\label{int}\emph{The 
First Reduction Case fails for $\vsim^{\text{{\bf R}}_i}_{\text{{\bf S}}_3}$. Take $\Gamma$ to be the set containing the two sentences  (the first one in the in the pure language language of equality) that say `there is at most one element in the universe' and `the binary relation $R$ is non-empty'. Let $\phi$ be `the binary relation $R$ is an injective but not surjective total mapping on the domain'. 
Then we have $L_\phi \subseteq L_\Gamma$, $\Gamma \vsim^{\text{{\bf R}}_1}_{\text{{\bf S}}_3} \phi$ (and hence  $\Gamma\vsim^{\text{{\bf R}}_i}_{\text{{\bf S}}_3} \phi$ by (i)  of Proposition \ref{genR}) but  $\Gamma  \nvdash  \phi$. Take any model $\mathfrak{A}\models \Gamma$, so $\mathfrak{A}$ must be isomorphic  to $\langle \{0\}, \{(0,0)\} \rangle$. {But then $\langle \{0\}, \{(0,0)\} \rangle \subseteq \langle \mathbb{Z}, R_{f^{-1}} \rangle \models \phi$ where $f: \mathbb{N} \longrightarrow \mathbb{Z}$ is the usual bijection $f(x) = \frac{x}{2}$ for $x$ even and $f(x)=\frac{-(x+1)}{2}$ otherwise, and $R_{f^{-1}}$ the graph of $f^{-1}$.}  Furthermore, since $\Gamma\vdash \neg \phi$, this shows a failure of the Second Reduction Case. In addition,  this latter observation also shows the failure of the Characterization in terms of Consistency and its Refinements as clearly  $\phi$ is not consistent with $\neg \phi$ which in turn is trivially consistent with $\Gamma$.}
\end{Rmk}

\begin{Rmk} \label{use2}\emph{We can also see that the  First Reduction Case, the Second Reduction Case, the Characterization in terms of Consistency and its Refinements, all  fail for $\vsim^{\text{{\bf R}}_i}_{\text{{\bf S}}_2}$.    This time let $\Gamma$  be the set containing the two sentences that say `there is at most one element in the universe' and `the \emph{complement} of the binary relation $R$ is non-empty'. Let $\phi$ be `the complement of the binary relation $R$ is an injective but not surjective total mapping on the domain'.  Analogously to our reasoning in Remark \ref{int}, $\Gamma \vsim^{\text{{\bf R}}_1}_{\text{{\bf S}}_2} \phi$ but $\Gamma  \nvdash  \phi$. The key observation is that $\text{{\bf S}}_2$, in contrast to $\text{{\bf S}}_3$ does require that the interpretation of the binary relation $R$ stays fixed when we go the extended model, blocking our previous argument, but does not demand that from the complement  of the interpretation of $R$ (the relation defined by the formula $\neg R(x,y)$). Hence we can repeat our reasoning \emph{mutatis mutandis} with $\neg R(x,y)$ this time  in place of $R(x, y)$.
}
\end{Rmk}

{
\begin{Rmk}\label{dif}\emph{Consider $\vsim^{\text{{\bf R}}_1}_{\text{{\bf S}}_3}$.  Observe that if $\psi$ is consistent, we do not necessarily get that $\emptyset\vsim^{\text{{\bf R}}_1}_{\text{{\bf S}}_3} \psi$ in contrast to the propositional case \cite{Makinson}. For example let $\psi$ be  `there are at most five elements in the universe'. For this  sentence in the pure equality language we have that  $\emptyset \not \vsim^{\text{{\bf R}}_1}_{\text{{\bf S}}_3} \psi$ for otherwise every set would be a subset of a set with at most five elements. More importantly, we can obtain a  failure of compactness reasoning similarly. We will do it for $ \vsim^{\text{{\bf R}}_4}_{\text{{\bf S}}_3}$ and $ \vsim^{\text{{\bf R}}_3}_{\text{{\bf S}}_3}$ and observe that $ \vsim^{\text{{\bf R}}_1}_{\text{{\bf S}}_3}$ is just a special case of $ \vsim^{\text{{\bf R}}_4}_{\text{{\bf S}}_3}$.
To see this, first we must recall that it is well-known that the class of 2-colourable graphs is not finitely axiomatizable  in the  language $L_\text{Graph}$ containing one binary relation $E$ only (see, for example, \cite{cai}), however we can clearly write a sentence $\phi$ in an expanded language such that every  2-colourable graph can be expanded to a model of $\phi$ (simply examine the definition of 2-colourability and add unary predicates for the colours). Furthermore, in the language $L_\text{Graph}$,  2-colourable graphs can be axiomatized by the infinite set  $\Gamma$ of universal sentences `$E$ is symmetric and irreflexive', and `there is no cycle of length $n$' for every odd $n\geq 3$. Then  $\Gamma \vsim^{\text{{\bf R}}_4}_{\text{{\bf S}}_3}\phi$, $\Gamma \vsim^{\text{{\bf R}}_3}_{\text{{\bf S}}_3}\phi$ and $\Gamma \vsim^{\text{{\bf R}}_1}_{\text{{\bf S}}_3}\phi$. Let $(\phi)_\forall$ denote the set of all universal logical consequences of  $\phi$ \emph{in the language $L_\Gamma$}. We can observe that $\Gamma \subseteq (\phi)_\forall$.  Thus there is no finite $\Gamma_0 \subseteq \Gamma $ such that   $\Gamma_0 \vdash (\phi)_\forall$ for otherwise $\Gamma_0$ finitely axiomatizes $\Gamma$ and hence the class of 2-colourable graphs, which is impossible.   Let us show that  for no  finite $\Gamma_0 \subseteq \Gamma $ it is the case that $\Gamma_0 \vsim^{\text{{\bf R}}_4}_{\text{{\bf S}}_3} \phi$ (or that $\Gamma_0 \vsim^{\text{{\bf R}}_3}_{\text{{\bf S}}_3} \phi$).
 Suppose that   $\Gamma_0 \vsim^{\text{{\bf R}}_4}_{\text{{\bf S}}_3} \phi$, then every model $\mathfrak{A}$ of $\Gamma_0$ for the language $L_{\Gamma_0}$ is  elementarily equivalent in the language $L_{\Gamma_0}=L_\Gamma$ to a model $\mathfrak{A}'$ that is a submodel of  $(\mathfrak{A}'' \upharpoonright L_{\Gamma_0})$ for a model $\mathfrak{A}''$ of $\phi$ for the language $L_{\Gamma_0, \phi}$. Obviously,  $\mathfrak{A}'' \models (\phi)_\forall$ and hence $(\mathfrak{A}'' \upharpoonright L_{\Gamma_0})\models (\phi)_\forall$ so $\mathfrak{A}' \models (\phi)_\forall$  since the sentences  in $(\phi)_\forall$ are all universal\cite[Corollary 2.4.2]{Hodges}. Thus by elementary equivalence, $\mathfrak{A} \models (\phi)_\forall$. Consequently, $\Gamma_0 \vdash (\phi)_\forall$, which is a contradiction. For the case of $\vsim^{\text{{\bf R}}_1}_{\text{{\bf S}}_3}$ simply observe that using Proposition \ref{genR} (i), we must have that   no  finite $\Gamma_0 \subseteq \Gamma $ it is the case that $\Gamma_0 \vsim^{\text{{\bf R}}_1}_{\text{{\bf S}}_3} \phi$.}
\end{Rmk}}

\begin{Rmk}\label{use4}\emph{Using Proposition \ref{genR},  the same argument as in Remark \ref{dif} works to refute compactness for $\vsim^{\text{{\bf R}}_i}_{\text{{\bf S}}_2}$ since $\Gamma_0 \not \vsim^{\text{{\bf R}}_4}_{\text{{\bf S}}_3} \phi$ implies $\Gamma_0 \not \vsim^{\text{{\bf R}}_4}_{\text{{\bf S}}_2} \phi$ and clearly $\Gamma \vsim^{\text{{\bf R}}_1}_{\text{{\bf S}}_2} \phi$.
}
\end{Rmk}

Table 1 collects the main facts that we have verified for variants of Definition \ref{friend} in first-order logic with equality {(it omits the data-points for first-order logic without equality that have been mentioned in passing in Remarks  \ref{comp2} and \ref{inter3})}. It remains incomplete, as the question marks flag cases that are still open.  The authors suspect that for those cells the answers are negative. {But, even in the unlikely case there is a row where both answers turn out positive,} the configuration $(\text{{\bf R}}_3, \text{{\bf S}}_1)$ that we have adopted in Definition  \ref{friend}, whose row in the table is highlighted by double lines, would appear to be the version of first-order friendliness that best preserves properties established for the propositional context.

\begin{table}[]
\centering
\bgroup
\def\arraystretch{2}
\begin{tabular}{|l|l|l|l|l|l|l|}
\hline
                                                                      & Supra & Reduction & Consistency & Compactness & Interpolation & Beth \\ \hline

{ $\vsim^{\text{{\bf R}}_1}_{\text{{\bf S}}_1} \ = \ \vsim^{\text{{\bf R}}_2}_{\text{{\bf S}}_1}$} &        $+$  [Rmk. \ref{use1}]      &         $+$ [Rmk. \ref{use1}]      &        $-$          [Rmk. \ref{use3}]              &      $-$   [Rmk. \ref{comp2}]      &      $-$ [Rmk. \ref{inter3}]          &    $+$ [Rmk. \ref{count}]   \\  \hline

{$\vsim^{\text{{\bf R}}_1}_{\text{{\bf S}}_2} \ = \  \vsim^{\text{{\bf R}}_2}_{\text{{\bf S}}_2}$ }&     $+$     [Rmk. \ref{use1}]       &      $-$   [Rmk. \ref{use2}]     &       $-$      [Rmk. \ref{use2}]                     &        $-$    [Rmk. \ref{use4}]           &       ?          &  ?      \\ 
\hline

{ $\vsim^{\text{{\bf R}}_1}_{\text{{\bf S}}_3} \ = \  \vsim^{\text{{\bf R}}_2}_{\text{{\bf S}}_3}$ } &     $+$   [Rmk. \ref{use1}]         &   $-$   [Rmk. \ref{int}]        &      $-$   [Rmk. \ref{int}]         &     $-$  [Rmk. \ref{dif}]               &      ?           &    ?    \\
 \hline \hline

 {$\vsim^{\text{{\bf R}}_3}_{\text{{\bf S}}_1}\ = \   \vsim^{\text{{\bf R}}_4}_{\text{{\bf S}}_1}$} &  $+$   [Prop. \ref{supra}]    &      $+$    [Prop. \ref{red1}\& \ref{red2}]      &       $+$     [Prop. \ref{consis}]                     &         $-$   [Prop. \ref{comp}]   &                 $-$  [Prop. \ref{inter}] &  $+$    [Prop. \ref{beth}]   \\

 \hline \hline

 $\vsim^{\text{{\bf R}}_3}_{\text{{\bf S}}_2}$&   $+$   [Rmk. \ref{use1}]           &      $-$    [Rmk. \ref{use2}]      &                  $-$    [Rmk. \ref{use2}]            &       $-$   [Rmk. \ref{use2}]             &     ?            &    ?    \\
 \hline

 $\vsim^{\text{{\bf R}}_3}_{\text{{\bf S}}_3}$&    $+$    [Rmk. \ref{use1}]          &    $-$   [Rmk. \ref{int}]       &          $-$            [Rmk. \ref{int}]                 &  $-$     [Rmk. \ref{dif}]              &       ?          &   ?     \\
  
 \hline

   $\vsim^{\text{{\bf R}}_4}_{\text{{\bf S}}_2}$ &    $+$    [Rmk. \ref{use1}]          &       $-$    [Rmk. \ref{use2}]     &                    $-$    [Rmk. \ref{use2}]              &      $-$      [Rmk. \ref{use2}]           &          ?       &      ? \\
 \hline

   $\vsim^{\text{{\bf R}}_4}_{\text{{\bf S}}_3}$ &    $+$   [Rmk. \ref{use1}]          &      $-$   [Rmk. \ref{int}]     &         $-$               [Rmk. \ref{int}]               &        $-$   [Rmk. \ref{dif}]          &       ?         &    ?    \\
\hline
\end{tabular}
\egroup
\
\medskip
\medskip
\medskip

\caption{Summary of properties of friends of friendliness}
\end{table}

\bigskip

\end{document}